\documentclass[12pt]{article}

\usepackage[margin=1in]{geometry}
\usepackage{amsmath,amsfonts,graphicx,framed}
\usepackage{amssymb}
\usepackage{amsthm}
\usepackage{fancyhdr}
\usepackage{float}
\usepackage{caption}
\usepackage{comment}
\usepackage[utf8]{inputenc}
\usepackage{authblk}
\usepackage[mathscr]{euscript}
\usepackage{mathtools}
\usepackage{xcolor}

\usepackage[hypertexnames=false]{hyperref}

\theoremstyle{plain}
\newtheorem{thm}{Theorem}
\newtheorem{lemma}{Lemma}

\newtheorem{prop}{Proposition}
\theoremstyle{definition}

\newtheorem{rem}{Remark}

\numberwithin{equation}{section}
\numberwithin{thm}{section}
\numberwithin{prop}{section}
\numberwithin{lemma}{section}
\numberwithin{cor}{section}
\numberwithin{defn}{section}
\numberwithin{rem}{section}
\numberwithin{ex}{section}

\newcommand{\pd}{\partial}

\newcommand{\mbb}{\mathbb}

\newcommand{\ep}{\varepsilon}

\newcommand{\re}{\mbb R}

\newcommand{\eqal}[1]{\begin{equation}\begin{aligned}#1\end{aligned}\end{equation}}

\begin{document}

\title{Regularity for the Monge-Amp\`ere equation by doubling}

\author{Ravi Shankar and Yu Yuan}

\date{}

\maketitle

\begin{abstract}
    We give a new proof for the interior regularity of strictly convex solutions of the Monge-Amp\`ere equation.  Our approach uses a doubling inequality for the Hessian in terms of the extrinsic distance function on the maximal Lagrangian submanifold determined by the potential equation.
\end{abstract}

\section{Introduction}

\footnotetext[1]{\today}

In this note, we present a new proof of interior regularity for strictly convex viscosity solutions of the Monge-Amp\`ere equation in general dimension:
\eqal{
\label{ma}
F(D^2u)=\det D^2u=1.
}

\begin{thm}
\label{thm:MA}
Let $u$ be a strictly convex viscosity solution of \eqref{ma} on a domain $\Omega\subset\re^n$ for $n\ge 2$.  Then $u$ is smooth inside $\Omega$.
\end{thm}


In \cite{P64} and \cite{P78}, Pogorelov showed the Hessian estimate using the strictly convex solution as a cut-off function in a Bernstein-Pogorelov maximum principle argument; consequently interior regularity for strictly convex solutions of \eqref{ma} was derived. Recall that the generalized solution in the integral sense there is equivalent to the generalized solution in the viscosity sense for the Monge-Amp\`ere equation.  The singular solutions of Pogorelov illustrate that a condition, such as strict convexity, is necessary for regularity in three and higher dimensions.  Earlier on, Alexandrov obtained strict convexity for the two dimensional Monge-Amp\`ere equation in \cite{A}. 

\smallskip
Without the strict convexity condition in two dimensions, the interior Hessian estimate was achieved by Heinz \cite{H} using isothermal coordinates. More recently, new pointwise proofs of the two dimensional estimate were found by Chen-Han-Ou \cite{CHO} and Guan-Qiu \cite{GQ} using different test functions in the maximum principle argument.  Other two dimensional proofs follow from various works on the sigma-2 equation in higher dimensions and special Lagrangian equation in the Euclidean setting over the past 16 years. Note that Liu's partial Legendre proof \cite{L} uses the strict convexity in \cite{A}.

\smallskip
An integral proof of Hessian estimates for strictly convex solutions to \eqref{ma} was recently found in \cite{Y} using the maximal surface interpretation of the equation for potential $u$.  The Lagrangian or ``gradient'' graph in pseudo-Euclidean space $(x,Du(x))\subset (\re^n\times\re^n,2dxdy)$ is volume maximizing. By establishing a monotonicity formula in terms of the extrinsic distance $x\cdot Du(x)$ to the origin $(0,0)$, integral arguments along the lines of Trudinger's proof of the gradient estimate for the minimal surface equation yield the Hessian estimate.

\smallskip
Our argument uses the extrinsic distance in a ``doubling'' way.  It is natural to attempt Pogorelov's proof using the extrinsic distance $x\cdot Du$ as a cut-off function, instead of the strictly convex potential $u$.  
That argument degenerates near $x=0$, but we can establish an a priori doubling inequality, Proposition \ref{prop:doub}, which controls the Hessian on outer balls by its values on small inner balls, measured in the extrinsic distance.  Alexandrov-Savin partial regularity is stable under smooth approximation, so by placing the inner ball inside the smooth set, the doubling inequality propagates the regularity to the outer ball.  This argument is in a similar spirit to the recent work on the sigma-2 equation in four dimensions \cite{SY}, but requires the extrinsic distance rather than the Euclidean one, since otherwise Pogorelov's singular solutions would satisfy the doubling property.

\smallskip
Now with two applications of the extrinsic distance done to the Monge-Amp\`ere equation, it is natural to ask if the strict convexity condition can be replaced by an alternative involving only suitably defined extrinsic quantities.

\section{Extrinsic properties}

\subsection{Doubling inequality under extrinsic distance}
Taking the gradient of the both sides of the Monge-Amp\`{e}re equation%
\begin{equation}
\ln\det D^{2}u=0, \label{lnMA}%
\end{equation}
we have%
\begin{equation}%
{\textstyle\sum\limits_{i,j=1}^{n}}
g^{ij}\partial_{ij}\left(  x,Du\left(  x\right)  \right)  =0, \label{DlnMA}%
\end{equation}
where $\left(  g^{ij}\right)  $ is the inverse of the induced metric
$g=\left(  g_{ij}\right)  =D^{2}u$ on the graph $\left(  x,Du\left(  x\right)
\right)  \subset\left(  \mathbb{R}^{n}\times\mathbb{R}^{n},2dxdy\right)  $
(for simplicity of notation, we drop the $2$ in $g=2D^{2}u$). Because of
\eqref{lnMA} and \eqref{DlnMA}, the Laplace-Beltrami operator of the metric
$g$ also takes the non-divergence form $\bigtriangleup_{g}=%
{\textstyle\sum\limits_{i,j=1}^{n}}
g^{ij}\partial_{ij}.$ 

The following strong subharmonicity of Hessian $D^2u$ was found in 
\cite{Y}.
\begin{prop}[Jacobi inequality]
\label{prop:Jac}
Suppose $u$ is a smooth solution to $\det D^{2}u=1.$ Then%
\eqal{
\label{Jacb}
\bigtriangleup_{g}\ln\det\left[  I+D^{2}u\left(  x\right)  \right]  \geq
\frac{1}{2n}\ \left\vert \bigtriangledown_{g}\ln\det\left[  I+D^{2}u\left(
x\right)  \right]  \right\vert ^{2}%
}
or equivalently for $a\left(  x\right)  =\left\{  \det\left[  I+D^{2}u\left(
x\right)  \right]  \right\}  ^{\frac{1}{2n}}$%
\begin{equation}
\label{Jacobi}
\bigtriangleup_{g}a\geq2\frac{\left\vert \bigtriangledown_{g}a\right\vert
^{2}}{a}.
\end{equation}
\end{prop}

Using the Jacobi inequality \eqref{Jacobi}, we derive an a priori doubling estimate for the Hessian in terms of the extrinsic balls of smooth strongly convex function $u$,
$$
D^u_r(p)=\{x\in\Omega:(x-p)\cdot (Du(x)-Du(p))<r^2\}.
$$
Denote the extrinsic distance of the position vector
$\left(  x,Du\right)  $ to the origin by%
\[
z=x\cdot u_{x}=\left(  x_{1},\cdots,x_{n}\right)  \cdot Du.
\]
Then%
\begin{gather}
\left\vert \nabla_{g}z\right\vert ^{2}=%
{\textstyle\sum\limits_{i,j=1}^{n}}
g^{ij}\partial_{i}z\partial_{j}z=%
{\textstyle\sum\limits_{i,j,k=1}^{n}}
g^{ij}\left(  u_{i}+x_{k}u_{ki}\right)  \left(  u_{j}+x_{k}u_{kj}\right)
\nonumber\\
\overset{p}{=}%
{\textstyle\sum\limits_{i=1}^{n}}
g^{ii}\left(  u_{i}^{2}+x_{i}^{2}u_{ii}^{2}+2x_{i}u_{i}u_{ii}\right)
\geq4x\cdot u_{x}=4z,\label{gradientZ}\\
\bigtriangleup_{g}z=x\cdot\bigtriangleup_{g}u_{x}+u_{x}\cdot\bigtriangleup
_{g}x+2\left\langle \nabla_{g}x,\nabla_{g}u_{x}\right\rangle =2%
{\textstyle\sum\limits_{i,j=1}^{n}}
g^{ij}\partial_{i}x_{k}\partial_{j}u_{k}\nonumber\\
\overset{p}{=}2%
{\textstyle\sum\limits_{i=1}^{n}}
g^{ii}u_{ii}=2n, \label{LaplaceZ}%
\end{gather}
where at any fixed point $p,$ we assume that $D^{2}u$ is diagonalized, and we
used (\ref{DlnMA}) for $\bigtriangleup_{g}z.$  It is clear that \eqref{gradientZ} and \eqref{LaplaceZ} work for other centers as well, $z=\langle x-p,Du(x)-Du(p)\rangle$.

\begin{prop}[Doubling inequality]
\label{prop:doub}
If $u$ is a smooth solution of \eqref{ma} on domain $\Omega\subset\re^n$ with $D^u_{r_4}(p)\subset\subset\Omega$ and $a=\det(I+D^2u)^{1/2n}$, then for $r_1<r_2<r_3<r_4$,
\eqal{
\label{doub}
\sup_{D^u_{r_2}(p)}a(x)\le C(n,r_1,r_2,r_3)\sup_{D^u_{r_1}(p)}a(x).
}
\end{prop}

\begin{proof}
We form the following Korevaar type test function on $D^u_{r_3}(p)$:
$$
w(x)=\eta(x)a(x),\qquad \eta(x)=\left[\exp\left(\frac{r_3^2-z(x)}{h}\right)-1\right]_+,
$$
where $h=h(n,r_i)$ will be fixed below, and $z(x)=\langle x-p,\cdot Du(x)-Du(p)\rangle$.  Let $x=x_*$ be the maximum point of $w$.  
By $Dw(x_*)=0$,
\eqal{
\label{grad}
D\eta=-\frac{\eta}{a}Da.
}
By $D^2w(x_*)\le 0$, \eqref{grad}, and Jacobi inequality \eqref{Jacobi}, we obtain
\eqal{
\label{Hess}
0\ge \Delta_gw&=a\Delta_g\eta+2\langle\nabla_g \eta,\nabla_ga\rangle+\eta\Delta_ga\\
&=a\Delta_g\eta+\eta\left(\Delta_ga-2\frac{|\nabla_ga|^2}{a}\right)\\
&\ge a\Delta_g\eta\\
&=a\frac{\eta+1}{h^2}\left(-h\Delta_gz+|\nabla_gz|^2\right).
}
Combining this with \eqref{gradientZ} and \eqref{LaplaceZ}, we obtain at $x=x_*$,
$$
z(x_*)\le nh/2= r_1^2
$$
if $h=2r_1^2/n$.  The doubling estimate \eqref{doub} follows:
$$
\sup_{D^u_{r_2}(p)}a(x)\le \sup_{D^u_{r_1}(p)}\frac{e^{(r_3^2-z(x))/h}-1}{e^{(r_3^2-r_2^2)/h}-1}\sup_{D^u_{r_1}(p)}a(x)\le \frac{e^{r_3^2/h}-1}{e^{(r_3^2-r_2^2)/h}-1}\sup_{D^u_{r_1}(p)}a(x).
$$
\end{proof}

\subsection{Extrinsic ball topology}

We recall that a convex function $u$ lies above its tangent planes, while a strictly convex function only intersects a tangent plane at a single point, and a strongly convex function has strictly positive Hessian.  Its subdifferential $\pd u(p)$ at a point $p$ is the collection of slopes of such tangent planes at $p$.  The subdifferentials $\pd u(x)$ are each closed and are locally bounded in $x$ as subsets of $\re^n$.  They increase: $\langle x-p,y-q\rangle\ge 0$ if $y\in\pd u(x),q\in\pd u(p)$.  They also vary continuously with $x$ and $u$: \cite[Theorem 24.5]{R} implies if $u_k\to u$ uniformly on $B_2(0)$ and $x_k\to x\in \overline B_1(0)$, then for any $\ep>0$, there exists $k_0$ such that
\eqal{
\label{R24.5}
\pd u_k(x_k)\subset\pd u(x)+\ep B_1(0),\qquad k\ge k_0.
}
Given $p\in B_1(0)$ and convex function $v(x)$ on $B_2(0)$, we define the \textit{outer sections} by
\eqal{
\label{section}
S^v_r(p):=\Big\{x\in B_1(0):v(x)<v(p)+\sup_{y\in\pd v(p)}[(x-p)\cdot y]+r^2\Big\},
}
and the \textit{extrinsic balls} of a smooth such $v$ by
\eqal{
\label{extrinsic}
D^v_r(p):=\{x\in B_1(0):(x-p)\cdot(Dv(x)-Dv(p))<r^2\}.
}
Euclidean balls ``bound'' the extrinsic balls from below: if $x,p\in B_1(0)$, then
\eqal{
\label{Euc}
B_{r^2/M}(p)\subset D^v_r(p),\qquad M=1+2\|D v\|_{L^\infty(B_2(0))}.
}
Conversely, the sections ``bound'' the extrinsic balls from above:
\eqal{
\label{ext}
D^v_r(p)\subset S^v_r(p).
}
Indeed, by convexity,
$$
v(p)\ge v(x)+(p-x)\cdot Dv(x),\qquad x,p\in B_1(0),
$$
so if the nonnegative $(x-p)\cdot (Dv(x)-Dv(p))<r^2$, then
\begin{align*}
    v(x)&\le v(p)+(x-p)\cdot Dv(p)+(x-p)(Dv(x)-Dv(p))\\
    &<v(p)+(x-p)\cdot D v(p)+r^2.
\end{align*}

\begin{rem}
    The containment \eqref{ext} is still valid for general strictly convex functions, using a subgradient version of the definition \eqref{extrinsic} of extrinsic balls.
\end{rem}

The section $S^u_r(p)$ shrinks to center $p$ uniformly in terms of the center, as the height $r^2$ goes to $0$, provided $u$ is strictly convex.

\begin{lemma}
\label{lem:diam}
    Given $u$ strictly convex on $B_2(0)$,
    \eqal{
    \label{diam}
    \lim_{r\to 0}\sup_{|p|<1}\text{\emph{diam}}\,S^u_r(p)=0.
    }
\end{lemma}
\begin{proof}
    If not, then there exist $r_k\to 0$ and $p_k,x_k\in B_1(0)$ with $x_k\in S^u_{r_k}(p_k)$ and $|x_k-p_k|\ge \ep$ for some fixed $\ep>0$.  We assume $x_k,p_k\to x,p\in \overline B_1(0)$.  By $x_k\in S^u_{r_k}(p_k)$, we have
    \eqal{
    \label{lem1a}
    u(x)<u(p)+(x-p)\cdot y_k+o(1)_k,
    }
    where $y_k\in\pd u(p_k)$ is the maximizer in \eqref{section}.  By \eqref{R24.5}, we can find a subsequence where $y_k\to y\in \pd u(p)$, so \eqref{lem1a} combined with strict convexity and $|x-p|\ge \ep$ gives
    $$
    u(x)\le u(p)+(x-p)\cdot y<u(x).
    $$
    This contradiction completes the proof.
\end{proof}

The section upper ``bound'' \eqref{ext} is preserved under limits.
\begin{lemma}
\label{lem:approx}
    Let $u$ be a strictly convex function on $B_2(0)$ with convex $u_k\to u$ uniformly on $B_2(0)$. Then,  for any $0<\delta<1$ and $p\in B_1(0)$, there exists $k_0$ large enough such that for all $0<r<1$ and $k\ge k_0$,  we have $S^{u_k}_r(p)\subset S^u_{r+\delta}(p)$.
\end{lemma}
\begin{proof}
    If not, then there exists $\delta>0$, $r_k\to r\in[0,1]$, and $x_k\in B_1(0)$ with $x_k\to x\in\overline B_1(0)$,  $x_k\in S^{u_k}_{r_k}(p)$, but $x_k\notin S^u_{r_k+\delta}(p)$.  The last condition implies
    $$
    u(x)\ge u(p)+\sup_{y\in\pd u(p)}[(x-p)\cdot y]+(r+\delta)^2+o(1)_k,
    $$
    while $x_k\in S^{u_k}_{r_k}(p)$ implies
    $$
    u(x)<u(p)+(x-p)\cdot y_k+r^2+o(1)_k
    $$
    for some $y_k\in \pd u_k(p)$.  By \eqref{R24.5}, we can find a subsequence such that $y_k\to y\in \pd u(p)$, so
    $$
    u(x)\le u(p)+(x-p)\cdot y+r^2\le u(x)+r^2-(r+\delta)^2.
    $$
    This contradiction completes the proof.
\end{proof}

As a consequence of Lemmas \ref{lem:diam} and \ref{lem:approx}, we see that $D^{u_k}_r(p)\subset S^u_{2r}(p)\subset\subset B_1(0)$ if $r$ is small enough depending on $u$, and $k$ is sufficiently large depending on $u,r,$ and $p$.  In this case, open set $D^{u_k}_r(p)$ has smooth boundary for smooth strictly
convex $u_k$ and is star shaped with center $x=p$.

\section{Proof of Theorem \ref{thm:MA}}

\textit{Step 1: Approximation.}  
Given viscosity solution $u(x)$ on $B_3(0)$, we show that $u$ is smooth in a neighborhood of any given point in $B_3(0)$, say near $x=0$.  By solving the Dirichlet problem, we find smooth solutions $u_k\to u$ uniformly on $B_2(0)$ 
with $Du_k$ bounded on $B_2(0)$ uniformly in $k$.  
In \eqref{Euc}, we enlarge the constant $M$: 
\eqal{
\label{M}
M:=1+2\sup_k\|Du_k\|_{L^\infty(B_2(0))}.
}

\smallskip
\noindent
\textit{Step 2: Partial regularity.} 
By combining Alexandrov's theorem (convex functions a.e. twice differentiable) with Savin's small perturbation theorem \cite[Theorem 1.3]{S}, the singular set $\text{sing}(u)$ is closed and measure zero.  By Savin's small perturbation theorem again, $u_k\to u$ in $C^{2,\alpha}_{loc}$ inside the open set $\text{sing}(u)^c:=B_2(0)\setminus\text{sing}(u)$, so the Hessians of $u_k$ are uniform in fixed compact subsets of $\text{sing}(u)^c$.

\smallskip
\noindent
\textit{Step 3: Uniform radius of the doubling balls.} We first find $r_2$ to control $x=0$ by $x=p$.  We first use \eqref{diam} to find $0<\rho<1$ small enough such that $S^u_{4\rho}(p)\subset\subset B_1(0)$ for all $|p|<1/2$.  Next we observe that $0\in B_{r^2/M}(p(r,e))$ for $p(r,e)=(r^2/2M)e$ and any $e\in S^{n-1}$.  Since $\text{sing}(u)$ is measure zero, we can fix $0<r\le \rho$ and $e\in S^{n-1}$ such that $p=p(r,e)\in \text{sing}(u)^c$.  This ensures $0\in B_{r^2/M}(p)\subset D^{u_k}_r(p)$, by \eqref{Euc}.  We set $r_2=r,r_3=2r$, and $r_4=3r$.  For the inner radius, since $p\in \text{sing}(u)^c$, we use \eqref{diam} to find $r_1<r_2$ small enough such that $S^u_{2r_1}(p)\subset\subset  \text{sing}(u)^c$.  

\smallskip
\noindent
\textit{Step 4: Doubling inequality.}  Now with $p$ and $r$ fixed, for $k$ large enough, Lemma \ref{lem:approx} and \eqref{ext} show $D^{u_k}_{r_4}(p)\subset S^{u_k}_{3r}(p)\subset S^{u}_{4r}(p)\subset\subset B_1(0)$, the equation domain $\Omega$.  We apply Proposition \ref{prop:doub} to obtain
$$
\max_{D^{u_k}_{r_2}(p)}|D^2u_k|\le C(r,r_1,n,\max_{D^{u_k}_{r_1}(p)}|D^2u_k|).
$$
For large enough $k$, we have $D^{u_k}_{r_1}(p)\subset S^{u_k}_{r_1}(p)\subset S^u_{2r_1}(p)\subset\subset\text{sing}(u)^c$.  By Alexandrov-Savin locally uniform convergence in $C^{2,\alpha}$ of $u_k$ to $u$ in $\text{sing}(u)^c$, we conclude a uniform Hessian bound in a fixed neighborhood of $x=0$, if $k$ is large enough:
$$
\max_{B_{r^2/M}(p)}|D^2u_k|\le \max_{D^{u_k}_{r_2}(p)}|D^2u_k|\le C(r,n,\max_{S^u_{2r_1}(p)}|D^2u|).
$$
By Calabi \cite{Cal} or Evans-Krylov, a subsequence $u_k$ converges in $C^{2,\alpha}_{loc}(B_{r^2/M}(p))$.  We conclude $u$ is smooth inside $B_{r^2/M}(p)$, hence smooth near $x=0$.  This completes the proof of interior regularity.

%
%
%
%
%
%
%
%

\bigskip
\textbf{Acknowledgments.}  Y.Y. is partially supported by an NSF grant.

\smallskip
\noindent
DEPARTMENT OF MATHEMATICS, PRINCETON UNIVERSITY, PRINCETON, NJ 08544-1000

\textit{Email address:} rs1838@princeton.edu

\smallskip
\noindent
DEPARTMENT OF MATHEMATICS, UNIVERSITY OF WASHINGTON, BOX 354350, SEATTLE, WA 98195

\textit{Email address:} yuan@math.washington.edu

\end{document}